\documentclass{elsarticle}

\usepackage{hyperref}
\usepackage{mathtools,amsmath}
\usepackage{amsthm}
\usepackage{amssymb}
\usepackage{color}
\usepackage{esint}

\newtheorem{thm}{Theorem}

\newtheorem{lem}{Lemma}
\newtheorem{proposition}{Proposition}
\newdefinition{definition}{Definition}	
\newdefinition{rmk}{Remark}

\newcommand{\norm}[1]{\Vert #1 \Vert}

\journal{}

\begin{document}
	
	\begin{frontmatter}
		
		\title{ Global boundedness and blow-up in a repulsive chemotaxis-consumption system in higher dimensions}
		
		\author[1]{Jaewook Ahn}
		\ead{jaewookahn@dgu.ac.kr}
		
		\author[2]{Kyungkeun Kang}
		\ead{kkang@yonsei.ac.kr}

		\author[3]{Dongkwang Kim\corref{correspondingauthor}}
		\cortext[correspondingauthor]{Corresponding author}
		\ead{dkkim@unist.ac.kr}
\address[1]{Department of Mathematics, Dongguk University, Seoul, Republic of Korea}
\address[2]{School of Mathematics \& Computing(Mathematics), Yonsei University, Seoul, Republic of Korea}
\address[3]{Department of Mathematical Sciences, Ulsan National Institute of Science and Technology (UNIST), Ulsan, Republic of Korea}

\begin{abstract}
This paper investigates the repulsive chemotaxis-consumption model
\begin{align*}
\partial_t u &= \nabla \cdot (D(u) \nabla u) + \nabla \cdot (u \nabla v), \\
0 &= \Delta v - uv,
\end{align*}
in an \(n\)-dimensional ball, \(n \ge 3\), where the diffusion coefficient \(D\) is an appropriate extension of the function \(0\le\xi\mapsto(1+\xi)^{m-1}\) for some \(m>0\). 
Under the  boundary conditions
\[
\nu \cdot (D(u) \nabla u + u \nabla v) = 0
\quad\text{ and }\quad v = M>0,\]
 we first demonstrate that for \(m > 1\), or \(m = 1\) with \(0 < M < 2/(n-2)\), the system admits globally defined classical solutions that are uniformly bounded in time for any choice of sufficiently smooth radial initial data. This result is further extended to the case \(0<m<1\) when \(M\) is chosen to be sufficiently small, depending on the initial conditions. In contrast, it is shown that for \(0 < m < \frac{2}{n}\), the system exhibits blow-up behavior for sufficiently large \(M\).
\end{abstract}

		
	\end{frontmatter}
	\section{Introduction} 
	Since the pioneering work of Keller and Segel in the early 1970s \cite{keller_initiation_1970,keller_model_1971}, the mathematical investigation of chemotaxis -- the directed movement of organisms in response to chemical gradients -- has significantly enhanced the understanding of various biological mechanisms. 
	Of particular interest is the chemotactic behavior of aerobic bacteria that consume oxygen as a nutrient, a subject that has attracted considerable attention due to its complex dynamics observed in both experimental and simulation settings \cite{dombrowski_self-concentration_2004, tuval_bacterial_2005}. 
		Accordingly, the following equations have been derived and subjected to rigorous analysis as a  fundamental model of this nutrient-based chemotaxis interaction \cite{lankeit_depleting_2023}:
	\begin{equation}\label{eqn_consumption}
	\begin{split}
	\partial_t u &= \nabla\cdot(D(u)\nabla u) -\nabla\cdot(uS(u,v)\nabla v),\\
	\partial_t v&= \Delta v -uf(v),
	\end{split}
	\end{equation}
where \(u\) denotes the density of the organisms and \(v\) stands for the concentration of nutrients, such as oxygen.

When \(D\), \(S\), and \(f\) are positive and appropriately chosen, it is well-known that the Neumann problem for \eqref{eqn_consumption} possesses an energy dissipation structure, arising from the interplay between the consumption term \(-uf(v)\) and the chemotaxis term \(-\nabla\cdot(uS(u,v)\nabla v)\), of the form in its simplest prototype case where \(D(u)=1\), \(S(u,v)=1\), and \(f(v)=v\):\begin{equation}\label{energy_inq}
\frac{d}{dt}\left\{\int_\Omega u\ln u +2\int_\Omega |\nabla\sqrt{v}|^2\right\}+\int_\Omega\frac{|\nabla u|^2}{u}+\int_\Omega v|\nabla^2\ln v|^2  \le 0.
\end{equation}
The energy dissipation structure \eqref{energy_inq} has played a crucial role in demonstrating global-in-time boundedness, stability, and asymptotic behaviors of solutions as well as global solvability (c.f. \cite{lankeit_depleting_2023}). 

Specifically, beyond the small initial data case where the global existence and asymptotic behaviors converging to a constant steady state have been proved using weighted \(L^p\)-estimates \cite{tao_boundedness_2011, zhang_stabilization_2015}, global large-data classical and weak solutions have been achieved via exploiting \eqref{energy_inq} in the two and three dimensions, respectively \cite{winkler_global_2012} (see \cite{jiang_eventual_2019,tao_eventual_2012} for eventual smoothness and asymptotic behaviors of such weak solutions in three dimensions).

Moreover, variants of the dissipation structure \eqref{energy_inq}, capable of allowing singular behaviors of \(S(v)=\frac{1}{v^\alpha}\), have been employed to see the global solvability and asymptotic behaviors of solutions. 
Indeed, the global existence of the classical solution has been proved for general \(f\) and \(S\) involving \(f(v)=v\) and \(S(v)=\frac{1}{v^\alpha}\) with \(0<\alpha<-\frac{1}{2}+\frac{1}{2}\left(\frac{8\sqrt{2}+11}{7}\right)^\frac{1}{2}\cong 0.3927\) for \(n=2\) \cite{ahn_global_2021}. Later, this result has been improved to \(0<\alpha<-\frac{1}{2}+\frac{1}{2}\sqrt{\frac{n+8}{n}}\) for \(n=2, 3\) \cite{kim_global_2023}. 
For \(0<\alpha<1\) and \(n=2\), weak solutions have been constructed when the initial data is sufficiently small \cite{viglialoro_global_2019}.
 In the case of \(\alpha=1\), another dissipation structure that provides weak regularity than \eqref{energy_inq} has been employed to prove global solvability of generalized solutions for \(n=2\) \cite{winkler_two-dimensional_2016} (c.f. \cite{lankeit_global_2020}), and of renormalized solutions for \(n\ge2\) \cite{winkler_renormalized_2018}. 

Furthermore, similar quantitative results have been considered when \(D\) is of porous medium type, \(u^{m-1}\), or its nondegenerate form, \((u+1)^{m-1}\) with \(m>1\), by using this enhancing diffusivity or  variants of \eqref{energy_inq}. 
Exploiting the diffusivity of \(D\), in presence of general \(S\ge 0\) and \(f>0\), the existence of global bounded solutions has been shown for \(m>1\) and \(n=2\) in a weak sense \cite{tao_global_2012}, and for \(m>2-\frac{2}{n}\) and \(n\ge2\) in the classical sense \cite{wang_boundedness_2014}.
In the case of \(S=1\), the global solvability has been shown for \(m>2-\frac{6}{n+4}\) and \(n\ge3\) using a variant of \eqref{energy_inq} \cite{wang_global_2015}. Later, this result has been improved by \cite{fan_global_2017}, where the global well-posedness and asymptotic behaviors are established for \(m>\frac{3}{2}-\frac{1}{n}\) and \(n\ge3\) \cite{fan_global_2017}. In the case of \(S(v)=\frac{1}{v}\), the global solvability has been proven for \(m>1+\frac{4}{n}\) and \(n\ge2\) \cite{lankeit_locally_2017}, and for \(m>\frac{3}{2}-\frac{1}{n}\) and \(n\ge2\) in a generalized concept \cite{yan_global_2018}.  
The latter has been extended to the tensor-valued \(S(u,v,x)\) satisfying \(|S(u,v,x)|\le \frac{1}{v^\alpha}\) with \(\alpha\in[0,1)\) for \(m> \frac{3}{2}-\frac{1}{n}\)  \cite{winkler_approaching_2022}.

For parabolic-elliptic simplification of \eqref{eqn_consumption} with \(D=1\), \(S=1\), and \(f(v)=v\), the global well-posedness and asymptotic behavior have been shown under Robin boundary condition by Fuest-Lankeit-Mizukami \cite{fuest_long-term_2021} when \(n\ge1\) (see also \cite{ahn_regular_2023}). We refer to \cite{yang_long_2024} for its Dirichlet counterpart.

Remarkably, the parabolic-elliptic counterpart of \eqref{eqn_consumption} may exhibit blow-up behavior when $S$ is of repulsion type, $S<0$. 
  Indeed, under No-flux/Dirichlet boundary conditions, finite-time blow-up solutions have been found for \(D(u)\le(u+1)^{m-1}\), \(0<m<1\), and \(S(v)=-\frac{1}{v}\) by Wang-Winkler \cite{wang_finite-time_2023} when \(n\ge2\). This result has also been verified for \(S=-1\) and \(n=2\) \cite{Ahn:2023aa}. Additionally, in the latter case, the global existence of bounded solutions has been established when \(D\) has a positive lower bound. However, to the best of our knowledge, there are no results concerning the existence or blow-up of solutions for \(S=-1\) and \(n\ge3\).
\,\\

Our main objective in this paper is to either establish the global existence of bounded solutions or construct unbounded solutions for the repulsion-type chemotaxis-consumption system in dimensions three and higher.
 To be more precise, we consider the following initial-boundary value problem
\begin{equation}\label{main}
\begin{cases}
	\partial_t u = \nabla\cdot(D(u)\nabla u) + \nabla\cdot(u\nabla v),&x\in\Omega,\,t>0,\\
	\quad 0= \Delta v -uv,&x\in\Omega,\,t>0,\\
	u(\cdot,t)\rvert_{t=0}=u_0,&x\in\Omega,\\
	\nu\cdot(D(u)\nabla u+ u\nabla v)=0\quad\text{and}\quad v=M,&x\in \partial\Omega,\,t>0,
\end{cases}
\end{equation}
In our analysis, the spatial domain \(\Omega\) is specifically chosen as \(B_R\), a ball of radius \(R>0\) centered at the origin in \(\mathbb{R}^n\). The diffusion coefficient \(D\) is chosen to extend the prototype choice as follows: 
\[D(\xi)=(\xi+1)^{m-1}\quad\text{for} \quad\xi\ge0\,\text{ and }\,m>0.\]
 Throughout this paper, the initial function \(u_0\) is assumed to satisfy 
\begin{equation}\label{assume_initial}
0 \not\equiv u_0 \in W^{1,\infty}(\Omega), \quad u_0 \geq 0, \quad \text{and is radially symmetric}.
\end{equation}
	
Before presenting the main results, we begin by outlining the local existence and blow-up criteria for \eqref{main}. Although these topics have been widely explored in the literature for more generalized frameworks \cite{cieslak_finite-time_2008}, this paper specifically addresses our particular case.

\begin{proposition}\label{prop_local}
Let \(n\ge3\), \(\Omega=B_R\subset\mathbb{R}^n\) with \(R>0\), and \(M>0\). Assume that 
\begin{equation}\label{assume_D}
D \in C^2([0, \infty)) \text{ fulfills } D(\xi) > 0 \text{ for all } \xi \ge 0,
\end{equation}
and that the initial function \(u_0\) satisfies \eqref{assume_initial}.
Then, there exists a maximal time \(T_{max}\in(0,\infty]\) for which the problem \eqref{main} possesses a unique classical solution \((u,v)\) in \(\Omega\times(0,T_{max})\), which is radially symmetric and positive over \(\overline\Omega\times(0,T_{max})\), with the following regularity properties:
	\begin{equation*}
	\begin{cases}
		u\in \bigcup_{q>n}C([0,T_{max}) ;W^{1,q}(\Omega))\cap C^{2,1}(\overline\Omega\times(0,T_{max})),\\
		v\in C^{2,0}(\overline\Omega\times(0,T_{max})).\\
	\end{cases}
\end{equation*}
Furthermore, if \(T_{max}<\infty\), then \(u\) blows up in the sense that
\begin{equation*}
\|u(\cdot,t)\|_{L^\infty(\Omega)} \rightarrow \infty \quad \textrm{as } t \rightarrow T_{max}.
\end{equation*}
Additionally, the mass of \(u\) is conserved over time:
\begin{equation}\label{prop_u}
\int_\Omega u(x,t) dx=\int_\Omega u_0(x) dx\quad\text{for all }t\in(0,T_{max}),
\end{equation}
and \(v\) is uniformly bounded in time:
\begin{equation}\label{prop_v}
\|v(\cdot,t)\|_{L^\infty(\Omega)}\le M\quad\text{for all }t\in(0,T_{max}).
\end{equation}
\end{proposition}
The first objective of this study is to investigate conditions on the diffusion rate that ensure the existence of globally bounded regular solutions. The following theorem demonstrates that, within this radially symmetric setting, the existence of a global-in-time bounded solution is assured for certain values of \(m\), specifically \(m>1\) or \(m=1\) with a constraint on \(M\).
	\begin{thm}\label{thm1}
		Let \(n\ge3\), \(\Omega=B_R\subset\mathbb{R}^n\) with \(R>0\), \(M>0\), and \(k_D>0\). Assume that \(D\) satisfies \eqref{assume_D} and  \begin{equation}\label{assume_D1}
		 D(\xi)\ge k_D(1+\xi)^{m-1}\quad\text{ for all }\,\xi>0,
		 \end{equation}
 and that either 
		 \[m>1,\quad \text{ or }\quad 
		 m=1\,\text{ and }\, M<\frac{2}{n-2}k_D\] 
		holds.
		 Then, for any choice of the initial function satisfying \eqref{assume_initial}, the problem \eqref{main} admits a classical radial solution \((u,v)\) in \(\Omega\times(0,\infty)\) which is uniformly bounded in time in a way that
		 \begin{equation*}
		 \sup_{t\ge0}\norm{u(\cdot,t)}_{L^\infty(\Omega)}\le C	 
		 \end{equation*}
with some \(C>0\). 
		\end{thm}

A subsequent natural question is under what conditions the boundedness or unboundedness of the solution is determined for the case \(0<m<1\).  It might be expected to experts that small data implies the existence of global bounded solutions. However, we can not find it in the literature and therefore, for clarity, we first demonstrate that bounded solutions can be constructed globally in time when the Dirichlet data $M$ is sufficiently small.

	\begin{thm}\label{thm2}
	Let \(n\ge3\) and \(\Omega=B_R\subset\mathbb{R}^n\) with \(R>0\). Assume that 
	\begin{equation}\label{assume_D2}
	D(\xi)=(1+\xi)^{m-1}\quad\text{ for all }\,\xi>0\quad\text{ with }\,0<m<1.
	\end{equation}
	Then, given initial function \(u_0\) satisfying \eqref{assume_initial}, there exists \(M_*=M_*(R,u_0)>0\) such that for any choice of \(M\in(0, M_*]\), the corresponding problem \eqref{main} possesses a classical radial solution \((u,v)\) in \(\Omega\times (0,\infty)\) which is uniformly bounded in time in a way that
		 \begin{equation*}
		 \sup_{t\ge0}\norm{u(\cdot,t)}_{L^\infty(\Omega)}\le C	 
		 \end{equation*}
with some \(C=C(M)>0\).
	\end{thm}	
	
Finally, we establish that when \(M\) is sufficiently large, blow-up of the solution may occur, in particular, in a more restricted range \(m\in(0,\frac{2}{n})\).
\begin{thm}\label{thm3}
Let \(n\ge3\), \(\Omega=B_R\subset\mathbb{R}^n\) with \(R>0\), and \(K_D>0\). Assume that 
\(D\) satisfies \eqref{assume_D} and 
\begin{equation}\label{assume_D3}
D(\xi)\le K_D(1+\xi)^{m-1}\quad\text{ for all }\,\xi>0\quad\text{ with }\,0<m<\frac{2}{n}.
\end{equation}
Then, given initial function \(u_0\) satisfying \eqref{assume_initial}, there exists \(M^*=M^*(R, u_0)>0\) such that for any choice of \(M\ge M^*\), the corresponding solution \((u,v)\) of \eqref{main} blows up in finite or infinite time; that is, \begin{equation*}
\|u(\cdot,t)\|_{L^\infty(\Omega)} \rightarrow \infty \quad \textrm{as } t \rightarrow T_{max}.
\end{equation*}
In particular, when \(n=3\), the blow-up time is finite, i.e., \(T_{max}<\infty\).
\end{thm}

\begin{rmk}
We remind that in two dimensions, blow-up solutions have been constructed for any \(m\in(0,1)\) \cite{Ahn:2023aa}. We may also expect that the blow-up could occur in the range \(m\in[\frac{2}{n},1)\) for \(n\ge3\) as well. We leave it as an open question.
\end{rmk}

{\textbf{Plan of the paper}}\quad  Theorem~\ref{thm1} extends the result in \cite{Ahn:2023aa} to higher dimensions. Whereas the authors in \cite{Ahn:2023aa} exploited the uniform \(L^2\) smallness of \(\nabla v\) near the origin to ensure uniform \(L^1\) boundedness of \(u\ln u\), thereby achieving global boundedness of \(u\) 
 for \(n=2\), our proof begins with a pointwise estimate of \(\nabla v\) near the origin (Lemma~\ref{lem1b}). This estimate is derived from the uniform boundedness of \(v\) and the assumption of radial symmetry. Employing this approach, we can control the mass accumulation function \(U\) near the origin through a comparison method for the equation of \(U\) (Lemma~\ref{lem1d}). This control leads to an improved regularity of \(\nabla v\) (Lemma~\ref{lem1e}) near the origin, which, when combined with estimates away from the origin (Lemma~\ref{lem1a}), ensures global boundedness of \(u\) through \(L^p\)-estimates in higher dimensions (Lemma~\ref{lem1f}).

For Theorem~\ref{thm2}, we first exploit the smallness assumption on \(M\), depending on \(u_0\), to obtain the behavior of \(U\) via a comparison argument similar to that used in Lemma~\ref{lem1d} (Lemma~\ref{lem2a}).
This behavior of \(U\), in conjunction with a Hardy-type inequality (Lemma~\ref{lem2b}) and a one-dimensional variant of the Gagliardo-Nirenberg inequality, allows us to derive \(L^p\)-estimates for \(u\) (Lemma~\ref{lem2c}).

In Theorem~\ref{thm3}, the method follows the approach employed in \cite{Ahn:2023aa} (see also \cite{wang_finite-time_2023}).
This method involves tracking the time evolution of the quantity \(\phi(t):=\int_0^R s^{-\alpha} U(s,t)ds\) with appropriately chosen \(\alpha\).
 Obtaining a suitable lower bound for \(\nabla v\) is crucial in this process. The authors in \cite{Ahn:2023aa} provided this lower bound by appropriately estimating \(\nabla \ln v\) and \(v\), respectively. 
However, while the lower bound for \(\nabla \ln v\) can be efficiently estimated even in higher dimensions using a variant of the ODE comparison method (Lemma~\ref{lem3a}, and see \cite[Lemma 3.1]{wang_finite-time_2023} for the proof), demonstrating the polynomial decay of \(v\), which was straightforward in dimensions two \cite{Ahn:2023aa}, seems difficult to demonstrate for dimensions three and higher. As a result, it is unlikely to obtain \(\phi'(t)\gtrsim \phi(t)\) as in \cite{Ahn:2023aa}.
To overcome the difficulty, 
instead, we make full use of the second equation of \eqref{main} and the estimate of \(v\) at the spatial origin (Lemma~\ref{lem3b}) to obtain an appropriate ordinary differential inequality (ODI) for $\phi(t)$ (Lemma~\ref{lem3c}), which is a technical novelty compare to  \cite{Ahn:2023aa}.
It turns out that such ODI  satisfies \(\phi'(t)\gtrsim 1\), namely, $\phi(t)$ has 
a linear growth in time, which implies blow-up in finite or infinite time. On the other hand, in three dimensions, since it is shown that $\phi(t)$ is uniformly bounded, blow-up must occur in a finite time.
\\

For clarity throughout this paper, we abbreviate \(\|\cdot\|_{L^p(\Omega)}\) as $\|\cdot\|_p$ when no confusion arises. Constants are denoted by \(C_i\) with \(i=1,2,...\), where the subscript serves to distinguish these constants within each chapter. Additionally, \(C(x,y,...)\) indicates a dependency of the constant \(C\) on variables \(x,y\), and others. 
Finally, we denote by \(\sigma_n\) the measure of the \(n-1\) dimensional unit sphere \(\partial B_1\).

\section{Global boundedness for \(m\ge 1\)}
Hereafter, we fix \(n\ge3\), \(\Omega=B_R\subset \mathbb{R}^n\) with \(R>0\), and assume that \(D\) satisfies \eqref{assume_D}. we denote by \((u,v)\) the solution in \(\Omega\times(0,T_{max})\) as in Proposition~\ref{prop_local}, and denote \[L:=\frac{1}{\sigma_n}\int_{\Omega}u_0(x)dx=\int_0^R s^{n-1}u_0(s)ds.\]
Utilizing the radial symmetry, the equations in \eqref{main} can be rewritten as the following scalar equations with a slight abuse of notation \((u,v)(x,t)=(u,v)(|x|,t)\), where \(|x|=s\in(0,R)\) and \(t\in(0,T_{max})\):
\begin{equation}\label{eqn_ur}
		\partial_t u = s^{1-n}\partial_s(s^{n-1}D(u)\partial_s u) + s^{1-n}\partial_s(s^{n-1}u \partial_s v),
		\end{equation}
\begin{equation}\label{eqn_vr}
		0= s^{1-n}\partial_s(s^{n-1}\partial_s v)-uv.		
\end{equation}
Hereafter, we fix \(n\ge3\), \(\Omega=B_R\subset \mathbb{R}^n\) with \(R>0\), and assume that \(D\) satisfies \eqref{assume_D}.
\subsection{Estimates of \(\nabla v\)}
In the subsequent two lemmas, pointwise estimates for \(\nabla v\) are presented. The estimates are two-fold: (Lemma~\ref{lem1a}) away from the spatial origin and (Lemma~\ref{lem1b}) near the origin. The first estimate indicates that the origin is the only possible blow-up point for \(u\), while the second leads to Lemma~\ref{lem1d}, the behavior of \(u\) near the origin.

We first observe that \(\partial_s v\) is positive and bounded away from the origin.
\begin{lem}\label{lem1a}
Let $(u,v)$ be a solution given in Proposition~\ref{prop_local}.
It holds that \(\partial_s v\) is positive in \((0,R)\times(0,T_{max})\). Moreover, for any choice of \(\delta_0\in(0,R)\), we have
	\begin{equation*}\label{ineq_vr}
	\partial_s v(s,t)\le \delta_0^{-n+1}LM\quad\text{for all }(s,t)\in[\delta_0,R)\times(0,T_{max}).
	\end{equation*}
\end{lem}
\begin{proof}
Since \(u\) and \(v\) are positive, the positivity of  \(\partial_s v\) immediately follows from \eqref{eqn_vr}. Let us fix \(\delta_0\in(0,R)\). Through integration by parts, it holds from \eqref{eqn_vr} and \eqref{prop_v} that
\begin{align*}
	L\ge\int_0^s \rho^{n-1}u(\rho,t) d\rho
	&=\int_0^s \frac{\partial_\rho(\rho^{n-1}\partial_\rho v(\rho,t))}{v(\rho,t)} d\rho\\
	&=\int_0^s \frac{\rho^{n-1}(\partial_\rho v(\rho,t))^2}{v^2(\rho,t)}d\rho+\frac{s^{n-1}\partial_s v(s,t)}{v(s,t)}\\
	&\ge \delta_0^{n-1}\frac{\partial_s v(s,t)}{M}
\end{align*}
for all \((s,t)\in[\delta_0,R)\times(0,T_{max})\), as desired.
\end{proof}
The following lemma shows that  \(\partial_s v \lesssim 1/s\) near the origin uniformly in time. 
Our basic approach involves employing a Newtonian kernel-type test function \(\varphi(x)=|x|^{2-n}-R^{2-n}\).

\begin{lem}\label{lem1b}
Let $(u,v)$ be a solution given in Proposition~\ref{prop_local}.
For any \(C_*>1\), it follows that
\begin{equation*} 
\partial_s v(s,t)\le\frac{(n-2)MC_*}{s}\quad\text{for all }(s,t)\in(0,R(1-1/C_*)^\frac{1}{n-2}]\times(0,T_{max}).
\end{equation*}
\end{lem}
\begin{proof}
In light of the positivity of \(u\) and \(v\), it holds through integration by parts and the positivity of \(v\) that
\begin{align*}
0&\le\frac{1}{\sigma_n}\int_{\Omega\setminus B_s}\Delta v(x,t) (|x|^{2-n}-R^{2-n})dx\\
&=\int_s^R \partial_\rho(\rho^{n-1}\partial_\rho v(\rho,t))(\rho^{2-n}-R^{2-n})d\rho\\
&=(n-2)\int_s^R \partial_\rho v(\rho,t) d\rho-s^{n-1}(s^{2-n}-R^{2-n})\partial_s v(s,t)\\
&\le(n-2)M-s^{n-1}(s^{2-n}-R^{2-n})\partial_s v(s,t)
\end{align*}
for all \((s,t)\in (0,R)\times(0,T_{max})\). Let \(C_*>1\) be arbitrarily chosen. Since
\(
 s^{n-2}\le R^{n-2}(1-\frac{1}{C_*})
\)
implies
\(
s^{2-n}\le C_*(s^{2-n}-R^{2-n})
\) and \(\partial_s v\) is positive as shown in Lemma~\ref{lem1a}, we obtain the desired result.
\end{proof}

\subsection{On the mass accumulation function of \(u\)}
We now examine the mass accumulation function of \(u\) given by
\begin{equation}\label{AUG25}
U(s,t):=\frac{1}{\sigma_n}\int_{B_s} u(x,t) dx=\int_0^s \rho^{n-1}u(\rho,t)d\rho
\end{equation}
for $(s,t)\in[0,R]\times[0,T_{max})$.
Clearly, \(U\) is nonnegative and nondecreasing. Based on the observation that blow-up of the solution can only occur near the origin (Lemma~\ref{lem1a}), our analysis primarily focuses on the behavior of  \(U(s,t)\) near the origin. 
As a preparation, the following simple calculation is needed.

\begin{lem}\label{lem1c}
Let $(u,v)$ be a solution given in Proposition~\ref{prop_local} and $U$ be defined in \eqref{AUG25}. For any \(C_*>1\),  \[U\in C^0([0,T_{max}); C^1([0,R]))\cap C^{2,1}((0,R]\times(0,T_{max}))\] satisfies
\begin{equation}\label{eqn1c_1}
\begin{aligned}
\partial_t U(s,t)&\le \mathcal{L}(U(s,t))\\&:= s^{n-1}D(s^{-n+1}\partial_sU(s,t))\partial_s(s^{-n+1}\partial_s U(s,t))\\
&\quad+\chi(n-2) MC_*s^{-1}\partial_s U(s,t)
\end{aligned}
\end{equation}
for all \((s,t)\in(0,R(1-1/C_*)^\frac{1}{n-2}]\times(0,T_{max})\).
\end{lem}
\begin{proof}
The regularity of \(U\) is directly derived from the properties of \(u\) in Proposition~\ref{prop_local}.
From the definition of \(U\), it is easily seen from \eqref{eqn_ur} that 
\begin{equation}\label{eqn1c_2}
\partial_t U(s,t)= s^{n-1}D(s^{-n+1}\partial_sU(s,t))\partial_s(s^{-n+1} \partial_s U(s,t))+\partial_s U(s,t) \partial_s v(s,t)
\end{equation}
for all \((s,t)\in(0,R)\times(0,T_{max})\). Hence, \eqref{eqn1c_1} can be straightforwardly deduced from \eqref{eqn1c_2} and Lemma~\ref{lem1b}. 
\end{proof}                   
In the next lemma, we show that when $m=1$, the inequality \eqref{eqn1c_1} provides the behavior of \(U\) near the origin.
\begin{lem}\label{lem1d}
Let $(u,v)$ be a solution given in Proposition~\ref{prop_local} and $U$ be defined in \eqref{AUG25}. 
Assume that \(D\) satisfies \eqref{assume_D1} with some \(k_D>0\), and that 
\[m=1\quad\text{ and }\quad M<\frac{2}{n-2}k_D.\]
 Then, one can find \(\beta=\beta(n,k_D,M)\in(0,2)\) for which it holds that
\begin{equation}\label{eqn1d_1}
U(s,t)\le Z(s):=\frac{L}{\delta_1^{n-\beta}}s^{n-\beta}\quad\text{ for all }(s,t)\in(0,\delta_1)\times(0,T_{max})
 \end{equation}
 with some  \(\delta_1=\delta_1(\beta,R, M,u_0)\in(0,R)\).
\end{lem}
       
\begin{proof}
Given the condition \(M<2k_D/(n-2)\), we first fix $\beta\in\left(\frac{M (n-2)}{k_D},2\right)$,  choose $C_{*}\in \left(1, \frac{k_D \beta}{M (n-2)}\right)$,
and define \(\delta_1\) as
\[
\delta_1=\min\left\{R\left(1-\frac{1}{C_*}\right)^{\frac{1}{n-2}}, \left(\frac{nL}{\norm{u_0}_\infty}\right)^\frac{1}{n},\left(\frac{(LM)^n}{n\sigma_n(2-\beta)}\right)^\frac{1}{n(n-2)}\right\}.
\]
Let \((s,t)\in (0,\delta_1)\times(0,T_{max})\). By performing a fundamental calculation on \(Z(s)\), we obtain
\begin{equation*}
\begin{split}
\mathcal{L}(Z(s))&=s^{n-1}D(s^{-n+1}\partial_sZ(s))\partial_s(s^{-n+1}\partial_s Z(s))+M(n-2)C_* s^{-1}\partial_s Z(s)\\
&\le-\frac{L\beta (n-\beta)k_D}{\delta_1^{n-\beta}}s^{n-\beta-2}+M(n-2)C_*\frac{L(n-\beta)}{ \delta_1^{n-\beta}}s^{n-\beta-2}\\
&=\frac{L(n-\beta)k_D}{ \delta_1^{n-\beta}}\left[-\beta+\frac{M(n-2)C_*}{k_D}\right]s^{n-\beta-2}\\
&\le 0.
\end{split}
\end{equation*}
In addition, we note that \(U(s,t)\le Z(s)\) along the parabolic boundary of \((0,\delta_1)\times (0,T_{max})\). Indeed, it is trivial that \(U(0,t)=0=Z(0)\) and  \(U(\delta_1,t)\le L=Z(\delta_1)\) for all \(t\in(0,T_{max})\). Furthermore, at \(t=0\), we find that
\[
U(s,0)\le \frac{\norm{u_0}_\infty}{n}s^n\le\frac{\norm{u_0}_\infty}{n}\delta_1^{\beta} s^{n-\beta}=\frac{\norm{u_0}_\infty\delta_1^n}{nL} Z(s)\le Z(s)
\]
for all \(s\in(0,\delta_1)\). Here, the last inequality of the above holds due to the fact that  \(\norm{u_0}_\infty\delta_1^n\le nL\).  
We now define an auxiliary function \(Y(s,t)=e^{-t}(Z(s)-U(s,t))\) over the domain \((0,\delta_1)\times(0,T_{max})\).  Then, we have
\[
\partial_t Y(s,t)+Y(s,t)=-e^{-t}\partial_t U(s,t)\ge e^{-t}(\mathcal{L}(Z(s))-\mathcal{L}(U(s,t))).
\]
To prove the inequality in \eqref{eqn1d_1}, we employ a variant of the comparison principle: Let \(\epsilon>0\) and define the set
\[
\mathcal{A}_\epsilon:=
\left\{ 
t\in(0,T_{max}) : Y(s,t)\le -\epsilon \,\text{ for some } s\in(0,\delta_1)
\right\}.
\]
Assume, for the sake of contradiction, that \(\mathcal{A}_\epsilon \neq\emptyset\). Let 
\(
\inf \mathcal{A}_\epsilon =t_\epsilon.
\)
Then, \(t_\epsilon\in(0,T_{max})\) due to \(Y\ge 0\) on the parabolic boundary of \((0,\delta_1)\times(0,T_{max})\). Thus we can choose \(s_\epsilon\in(0,\delta_1)\) fulfilling \(Y(s_\epsilon,t_\epsilon)=-\epsilon\). Note that 
\(\partial_s Y(s_\epsilon,t_\epsilon)=0\), \(\partial_s^2 Y(s_\epsilon,t_\epsilon)\ge 0\), and \( \partial_t Y(s_\epsilon,t_\epsilon)\le 0\). Therefore, it holds that
\begin{align*}
-\epsilon&\ge \partial_t Y(s_\epsilon,t_\epsilon)+ Y(s_\epsilon,t_\epsilon)\\
&=D(s_\epsilon^{-n+1}\partial_s Z(s_\epsilon,t_\epsilon))\partial_s^2 Y(s_\epsilon,t_\epsilon)\ge 0,
\end{align*}
which contradicts \(\epsilon>0\).
This implies that \(\mathcal{A}_\epsilon=\emptyset\).
Since \(\epsilon>0\) can be taken arbitrarily small, it follows that \(Y(s,t)\ge 0\) for every \((s,t)\in(0,\delta_1)\times(0,T_{max})\), thereby proving \eqref{eqn1d_1}.
\end{proof}

\subsection{Proof of Theorem~\ref{thm1}}
The upper estimate of \(U\) established in Lemma~\ref{lem1d}  leads to a more refined behavior for \(\nabla v\) near the origin compared to that in Lemma~\ref{lem1b}.

\begin{lem}\label{lem1e}
Let $(u,v)$ be a solution given in Proposition~\ref{prop_local}.
Assume that \(D\) satisfies \eqref{assume_D1} with some \(k_D>0\), and that 
\[m=1\quad\text{ and }\quad M<\frac{2}{n-2}k_D.\]
Then, there exists \(\beta\in(0,2)\) such that for some \(C>0\)
\begin{equation}\label{eqn1e_0}
\partial_sv(s,t)\le Cs^{-\beta+1}\quad\text{for all } (s,t)\in(0,R)\times(0,T_{max}).
\end{equation}
 In particular, for any \(\epsilon>0\), we can find \(\delta=\delta(\epsilon)>0\) satisfying
\begin{equation}\label{eqn1e_0_1}
\norm{\nabla v(\cdot,t)}_{L^n(B_\delta)}\le \epsilon\quad\text{for all } t\in(0,T_{max}).
\end{equation}
\end{lem}
\begin{proof}
Note that it is enough to consider the case \(\epsilon\in(0,1)\). Let \(\epsilon\in(0,1)\). We fix \(\beta\in(0,2)\) and \(\delta_1\) as in Lemma~\ref{lem1d}, and then define 
\[ C_1=\left(\frac{(LM)^n}{n\sigma_n(2-\beta)}\right)^\frac{1}{n(n-2)}
\]
and
\[
\delta=\epsilon^\frac{1}{2-\beta}C_1^{-\frac{n-2}{2-\beta}} \delta_1^\frac{n-\beta}{2-\beta}.
\]
Since the definition of \(\delta_1\) yields \(\delta_1\le C_1\le C_1 \epsilon^{-\frac{1}{n-2}}\), it holds that 
\[\frac{\delta}{\delta_1}=\epsilon^\frac{1}{2-\beta}C_1^{-\frac{n-2}{2-\beta}}\delta_1^\frac{n-2}{2-\beta}\le 1,\]
that is, \(\delta\le\delta_1\).
 For fixed \((s,t)\in(0,\delta)\times(0,T_{max})\), it follows from \eqref{eqn_vr} and \eqref{prop_v} that 
\begin{equation}\label{eqn1e_1}
\partial_s v(s,t)=s^{-n+1}\int_0^s \rho^{n-1} uv ds\le M s^{-n+1}U(s,t)
\end{equation}
which, along with \eqref{eqn1d_1}, proves \eqref{eqn1e_0}. Furthermore, \eqref{eqn1e_0} leads to \eqref{eqn1e_0_1},
\begin{equation*}
\norm{\nabla v(\cdot,t)}_{L^n(B_\delta)}
=\left(\frac{1}{\sigma_n}\int_0^\delta s^{n-1}(\partial_s v)^n ds\right)^{1/n}
\le\frac{C_1^{n-2}}{\delta_1^{n-\beta}}\delta^{2-\beta}=\epsilon.
\end{equation*}
This completes the proof.
\end{proof}

We are ready to prove the boundedness of \(u\) in \(L^p\) for large \(p\).

\begin{lem}\label{lem1f}
Let $(u,v)$ be a solution given in Proposition~\ref{prop_local}.
Assume that \(D\) satisfies \eqref{assume_D1} with some \(k_D>0\), and that either
 \[m>1,\quad \text{ or }\quad  m=1\quad \text{and}\quad M<\frac{2}{n-2}k_D\] 
 holds. For any \(p>\max\left\{1,m-1\right\}\), there exists \(C=C(p)>0\) such that
\begin{equation}\label{eqn1f_1}
\sup_{t\in(0,T_{max})}\norm{u(\cdot,t)}_p\le C.
\end{equation}
\end{lem}
\begin{proof}
Let us fix \(p>\max\left\{1,m-1\right\}\). We first consider the case that \[m=1\quad\text{ and }\quad M<\frac{2}{n-2}k_D.\] Let \(\epsilon=\frac{k_D}{8pB}\), where \(B=B(n,\Omega)>0\) is the constant appeared in Sobolev embedding theorem
\[\norm{f}_\frac{2n}{n-2}\le B(\norm{\nabla f}_2+\norm{f}_2)\quad\text{ for all } f\in H^1(\Omega).\]
 For such \(\epsilon\), we take \(\delta>0\) as in Lemma~\ref{lem1e}. By applying Lemma~\ref{lem1a} with \(\delta_0=\delta\), we obtain that \(\partial_s v\) is bounded in \([\delta,R)\times(0,T_{max})\).
Thus, it holds that 
 \begin{equation}\label{eqn1f_2}
 \norm{\nabla v(\cdot,t)}_{L^n(B_\delta)}\le \epsilon\,\text{ and }\, \norm{\nabla v(\cdot,t)}_{L^\infty(\Omega\backslash B_\delta)} \le C_2\quad\text{for all }\,t\in(0,T_{max})
\end{equation}
with some \(C_2=C_2(\delta)>0\). For fixed \(t\in(0,T_{max})\), we now observe standard \(L^p\)-estimates for \(u\):
\begin{equation*}
\begin{split}
\frac{d}{dt}\norm{u(\cdot,t)}_p^p&+\frac{4k_D(p-1)}{p}\norm{\nabla u^\frac{p}{2}}_2^2
\\&\le 2(p-1)\norm{u^\frac{p}{2}\nabla u^\frac{p}{2}\cdot\nabla v}_1\\
&=2(p-1)\left(\norm{u^\frac{p}{2}\nabla u^\frac{p}{2}\cdot \nabla v}_{L^1(B_\delta)}+\norm{u^\frac{p}{2}\nabla u^\frac{p}{2}\cdot \nabla v }_{L^1(\Omega\backslash B_\delta)}\right)\\
&=:I_1+I_2.
\end{split}
\end{equation*}
By Sobolev embedding, the Gagliardo-Nirenberg inequality,  and Young's inequality, we have 
\begin{align*}
\norm{u^\frac{p}{2}}_\frac{2n}{n-2}&\le B(\norm{\nabla u^\frac{p}{2}}_2+\norm{u^\frac{p}{2}}_2)\\
&\le B(\norm{\nabla u^\frac{p}{2}}_2+C_3(\norm{\nabla u^\frac{p}{2}}_2^\theta\norm{u}_1^{\frac{p}{2}(1-\theta)}+\norm{u}_1^{\frac{p}{2}}))\\
&\le2B\norm{\nabla u^\frac{p}{2}}_2+C_4,
\end{align*}
where \(C_3>0\), \(C_4=C_4(p)>0\), and \(\theta=\frac{p-1}{p-1+2/n}\in(0,1)\). Hence, it follows from \eqref{eqn1f_2} and Young's inequality that for some \(C_5=C_5(p)>0\)
\begin{equation*}
\begin{split}
I_1&\le 2(p-1)\norm{\nabla u^\frac{p}{2}}_2\norm{u^\frac{p}{2}}_\frac{2n}{n-2}\norm{\nabla v}_{L^n(B_\delta)}\\
&\le8 B(p-1)\epsilon\norm{\nabla u^\frac{p}{2}}_2^2+C_5\\
&\le\frac{k_D(p-1)}{p} \norm{\nabla u^\frac{p}{2}}_2^2+C_5.
\end{split}
\end{equation*}
Similarly,
\begin{equation*}
\begin{split}
I_2&\le 2(p-1)\norm{|\nabla u^\frac{p}{2}| u^\frac{p}{2}}_1\norm{\nabla v}_{L^\infty(\Omega\backslash B_\delta)}\\
&\le 2(p-1)C_2\norm{\nabla u^\frac{p}{2}}_2\norm{u^\frac{p}{2}}_2\\
&\le 2(p-1)C_2C_3\norm{\nabla u^\frac{p}{2}}_2(\norm{\nabla u^\frac{p}{2}}_2^\theta\norm{u}_1^{\frac{p}{2}(1-\theta)}+\norm{u}_1^{\frac{p}{2}}) \\
&\le \frac{k_D(p-1)}{p}\norm{\nabla u^\frac{p}{2}}_2^2+C_6,
\end{split}
\end{equation*}
where \(C_6=C_6(p)>0\) and  \(\theta=\frac{p-1}{p-1+2/n}\in(0,1)\).
Therefore,  it follows from the above estimates that
\begin{equation}\label{eqn1f_3}
\frac{d}{dt}\norm{u(\cdot,t)}_p^p+\frac{2k_D(p-1)}{p}\norm{\nabla u^\frac{p}{2}}_2^2\le C(p)\quad\text{for all }\,t\in(0,T_{max}).
\end{equation}
Next, in the case of \[m>1,\] we use Lemma~\ref{lem1a} and Lemma~\ref{lem1b} to control \(\nabla v\).
In a similar way, let \(t\in(0,T_{max})\), and we see that
\begin{equation}\label{eqn1f_4}
\frac{d}{dt}\norm{u(\cdot,t)}_p^p+k_Dp(p-1)\int_\Omega (1+u)^{m-1}u^{p-2}|\nabla u|^2 dx
\le p(p-1)\norm{ u^{p-1}\nabla u\cdot\nabla v}_1.
\end{equation}
 We first note that 
\begin{equation}\label{eqn1f_5}
\frac{4k_Dp(p-1)}{(p+m-1)^2}\int_\Omega |\nabla u^\frac{p+m-1}{2}|^2 dx\le k_Dp(p-1)\int_\Omega (1+u)^{m-1}u^{p-2}|\nabla u|^2 dx.
\end{equation}
On the one hand, a simple calculation and H\"older's inequality yields
\begin{equation*}
\begin{split}
p(p-1)\norm{ u^{p-1}\nabla u\cdot\nabla v}_1&=\frac{2p(p-1)}{p+m-1}\norm{ u^\frac{p-m+1}{2}\nabla u^\frac{p+m-1}{2}\cdot\nabla v}_1\\
&\le \frac{2p(p-1)}{p+m-1} \norm{\nabla u^\frac{p+m-1}{2}}_2\norm{u^\frac{p-m+1}{2}}_q\norm{\nabla v}_{q'},
\end{split}
\end{equation*}
where \(q=\frac{(p+m-1)}{p-m+1}\cdot\frac{2n}{n-2}\) and \(q'\) is such that \(\frac{1}{q}+\frac{1}{q'}=\frac{1}{2}.\) 
By Sobolev embedding, it follows that 

\begin{align*}
\norm{u^\frac{p-m+1}{2}}_q&=\norm{u^\frac{p+m-1}{2}}_{\frac{2n}{n-2}}^\frac{p-m+1}{p+m-1}\\
&\le B(\norm{\nabla u^\frac{p+m-1}{2}}_2^\frac{p-m+1}{p+m-1}+\norm{u^\frac{p+m-1}{2}}_2^\frac{p-m+1}{p+m-1})\\
&\le B(\norm{\nabla u^\frac{p+m-1}{2}}_2^\frac{p-m+1}{p+m-1}\\
&\quad+C_7(\norm{\nabla u^\frac{p+m-1}{2}}_2^{\frac{p-m+1}{p+m-1}\theta}\norm{u}_1^{\frac{p-m+1}{2}(1-\theta)}+\norm{u}_1^{\frac{p-m+1}{2}}))\\
&\le2B\norm{\nabla u^\frac{p+m-1}{2}}_2^{\frac{2p}{p+m-1}-1}+C_8,
\end{align*}
where \(C_7>0\), \(C_8=C_8(p)>0\), and \(\theta=\frac{p+m-2}{p+m-2+2/n}\in(0,1)\).
From Lemma~\ref{lem1a} and Lemma~\ref{lem1b}, we can show that \(s\partial_s v\) is bounded in \((0,R)\times(0,T_{max})\), which induces that
\(\norm{\nabla v(\cdot,t)}_{q'}\) is bounded in \((0,T_{max})\) due to the fact that 
\[
\frac{1}{q'}=\frac{1}{2}-\frac{1}{q}=\frac{1}{2}-\frac{p-m+1}{p+m-1}\cdot\frac{n-2}{2n}>\frac{1}{2}-\frac{n-2}{2n}=\frac{1}{n}.
\]
Thus, noticing that \(\frac{2p}{p+m-1}<2\), we obtain through Young's inequality that
\begin{equation}\label{eqn1f_6}
p(p-1)\norm{ u^{p-1}\nabla u\cdot \nabla v}_1\le \frac{2k_Dp(p-1)}{(p+m-1)^2}\norm{\nabla u^\frac{p+m-1}{2}}_2^2+C_9
\end{equation}
 with some \(C_9=C_9(p)>0\).
 Hence, combining \eqref{eqn1f_5} and \eqref{eqn1f_6} into \eqref{eqn1f_4}, we see that 
 \begin{equation}\label{eqn1f_7}
\frac{d}{dt}\norm{u(\cdot,t)}_p^p+\frac{2k_Dp(p-1)}{(p+m-1)^2}\norm{\nabla u^\frac{p+m-1}{2}}_2^2
\le C_9 \quad\text{for all }\,t\in(0,T_{max}).
 \end{equation}
In addition, the Gagliardo-Nirenberg inequality, Young's inequality, and \eqref{prop_u} yield that for any \(\mu>0\) and \( m\ge1\), we can find \(C(p,\mu)>0\) such that
\[\norm{u^\frac{p}{2}}_2^2\le\mu\norm{\nabla u^\frac{p+m-1}{2}}_2^2+C(p,\mu).\]
Therefore, writing \(y(t):=\norm{u^\frac{p}{2}(\cdot,t)}_2^2\), both \eqref{eqn1f_3} and \eqref{eqn1f_7} can be rewritten by
\[y'(t)+y(t)\le C(p),\]
which induces \eqref{eqn1f_1}.
\end{proof}
\begin{proof}[\textit{Proof of Theorem~\ref{thm1}}]\quad
Note that the standard elliptic regularity theory and Lemma~\ref{lem1f} with \(p>n\) show \(\norm{\nabla v(\cdot,t)}_\infty\) is bounded uniformly in time. Thus, as in \eqref{eqn1f_4} and \eqref{eqn1f_5}, we have for some \(C>0\)
\begin{align*}
\frac{d}{dt}\norm{u(\cdot,t)}_p^p&+\frac{4k_Dp(p-1)}{(p+m-1)^2}\int_\Omega |\nabla u^\frac{p+m-1}{2}|^2 dx\\
&=-\frac{2p(p-1)}{p+m-1}\int_\Omega u^\frac{p-m+1}{2}\nabla u^\frac{p+m-1}{2}\cdot \nabla v dx \\
&\le \frac{2k_Dp(p-1)}{(p+m-1)^2}\int_\Omega |\nabla u^\frac{p+m-1}{2}|^2 dx+Cp(p-1)\int_\Omega u^{p-m+1} dx
\end{align*}
for all \(p>\max\left\{1,m-1\right\}\). Now, employing a standard iteration method (see e.g., \cite{alikakos1979lp, tao_boundedness_2012}), we can show that \(\norm{u(\cdot,t)}_\infty\) is uniformly bounded in \((0,T_{max})\), which proves Theorem~\ref{thm1} with the help of blow-up criteria in Proposition~\ref{prop_local}. 
\end{proof}

\section{Global boundedness for \(0<m<1\)}

\begin{lem}\label{lem2a}
Let $(u,v)$ be a solution given in Proposition~\ref{prop_local}. Assume that \(D\) satisfies \eqref{assume_D2}.
For any \(\gamma\in(0,\frac{2}{2-m})\), there exists \(M_*=M_*(\gamma,u_0)>0\) such that if \(M\le M_*\), then $U$ defined in \eqref{AUG25}
satisfies 
\begin{equation}\label{eqn2a_1}
U(s,t)\le W(s):=\eta s^{n-\gamma}\quad\text{ for all }(s,t)\in(0,R)\times(0,T_{max}), 
\end{equation}
where \(\eta=\eta(u_0,\gamma):=\sup_{r\in(0,R)}r^{-n+\gamma}\int_0^r \rho^{n-1}u_0(\rho) d\rho\).
\end{lem}
\begin{proof}
We begin by noting that \(\eta>0\) is finite since \(u_0\) satisfies \eqref{assume_initial}.
For fixed \((s,t)\in (0,R)\times(0,T_{max})\),
since \eqref{eqn1c_2} and \eqref{eqn1e_1} are still valid, we have
 \begin{equation*}
 \begin{aligned}
 \partial_t U(s,t)&\le\mathcal{H}(U(s,t))
 \\&:= s^{n-1}(1+s^{-n+1}\partial_sU(s,t))^{m-1}\partial_s(s^{-n+1} \partial_s U(s,t))\\&\quad+M s^{-n+1}U(s,t)\partial_s U(s,t).
 \end{aligned}
 \end{equation*}
We again use the comparison principle, as in the proof of Lemma~\ref{lem1d}; we first observe that  
\begin{align*}
\mathcal{H}(W(s,t))&=-\eta\gamma(n-\gamma) (1+\eta(n-\gamma)s^{-\gamma})^{m-1}s^{n-\gamma-2}
+M\eta^2(n-\gamma) s^{n-2\gamma}\\
&=\eta(n-\gamma)s^{n-2\gamma}\left[M\eta-\gamma(s^\gamma+\eta(n-\gamma))^{m-1}s^{\gamma(2-m)-2}
\right].
\end{align*}
Thus, for any \[M\le M_*:=\frac{\gamma}{\eta(R^\gamma+\eta(n-\gamma))^{1-m}R^{2-\gamma(2-m)}},\] we have \(\mathcal{H}(W(s,t))\le 0.\)
Furthermore, it follows from the definition of \(\gamma\) that for all \(t\in(0,T_{max})\) and \(r\in(0,R)\),
\[U(R,t)\le W(R)\,\text{ and }\, U(r,0)\le W(r),\,\text{ as well as }\,U(0,t)=0=W(0).\]
Consequently, as in the proof of Lemma~\ref{lem1d}, we can conclude \eqref{eqn2a_1}. 
\end{proof}
We now proceed to the standard \(L^q\)- estimates for \(u\). As a preliminary step, we introduce a simplified version of the Hardy inequality on balls. 
\begin{lem}\label{lem2b}
For any radially symmetric function \(f\in (H^1\cap L^l)(B_R)\), \(l>0\), there exists \(C>0\) such that
\begin{equation}\label{eqn2b_1}
\int_{B_R}\frac{(f(x))^2}{|x|^2}dx\le C(\norm{\nabla f}_2^2+ \norm{f}_l^2).
\end{equation}
\end{lem}
\begin{proof}
By applying integration by parts and Young's inequality, we can derive 
\begin{equation}\label{eqn2b_2}
\begin{aligned}
\int_0^R s^{n-3}f^2(s)ds&=-\frac{2}{n-2}\int_0^R s^{n-2}f(s)\partial_s f(s) ds+\frac{1}{n-2}R^{n-2}f^2(R)\\
&\le \frac{1}{2}\int_0^R s^{n-3}f^2(s)ds+C_1\int_0^R s^{n-1}(\partial_s f(s))^2 ds\\&\quad+\frac{1}{(n-2)R}\int_{\partial B_R} f^2(x)dS
\end{aligned}
\end{equation}
with some \(C_1>0\). Note that it is enough to consider the case \(l<2\). By utilizing the standard trace embedding \(W^{1,2}(\Omega)\hookrightarrow L^2(\partial\Omega)\) and subsequently applying the Gagliardo-Nirenberg inequality along with Young's inequality, we obtain
\begin{equation}\label{eqn2b_3}
\begin{aligned}
\norm{f}_{L^2(\partial B_R)}^2&\le C_2(\norm{\nabla f}_2^2+\norm{f}_2^2)\\
&\le C_3(\norm{\nabla f}_2^2+\norm{\nabla f}_2^{2\theta}\norm{f}_l^{2(1-\theta)}+\norm{f}_l^2)\\
&\le C_4(\norm{\nabla f}_2^2+\norm{f}_l^2),
\end{aligned}
\end{equation}
where \(C_i>0\), \(i=2,3,4\), and \(\theta=\frac{\frac{1}{l}-\frac{1}{2}}{\frac{1}{l}-\frac{1}{2}+\frac{1}{n}}\in(0,1)\). 
In consequence, by employing \eqref{eqn2b_2} and \eqref{eqn2b_3} together, we have \eqref{eqn2b_1}.
\end{proof}

\begin{lem}\label{lem2c}
Let $(u,v)$ be a solution given in Proposition~\ref{prop_local}.
Under the same assumptions as in Lemma~\ref{lem2a}, if $M\le M_*$, where  \(M_*\) is defined in Lemma~\ref{lem2a},  then
there exist $q$ with $\displaystyle q>\frac{(n-1)(2-m)}{m}+1-m$ and \(C=C(q)>0\)
such that
\begin{equation}\label{eqn2c_1}
\sup_{t\in(0,T_{max})}\norm{u(\cdot,t)}_q\le C.
\end{equation}
\end{lem}
\begin{proof}
For given \(\gamma\in(0,\frac{2}{2-m})\), we pick \(q>1\) such that 
\begin{equation}\label{eqn2c_2}
\frac{(n-1)(2-m)}{m}+1-m<q<\frac{n-1}{(\gamma-1)_+} +1-m.
\end{equation}
Then, we choose \(k>1\) to satisfy
\begin{equation}\label{eqn2c_3}
1-m<\frac{1}{k}<1-\frac{(n-1)(2-m)}{q+m-1}.
\end{equation}
For fixed \(t\in(0,T_{max})\), we multiply the first equation in \eqref{main} by \(q(u+1)^{q-1}\) and integrate by parts, then add \(\int_\Omega (u+1)^q dx\) to both sides to obtain
\begin{equation}\label{eqn2c_4}
\begin{aligned}
\frac{d}{dt}\int_\Omega (u+1)^q dx+&\int_\Omega(u+1)^q dx+\frac{4q(q-1)}{(m+q-1)^2}\int_\Omega |\nabla(u+1)^\frac{q+m-1}{2}|^2 dx\\
&=-q(q-1)\int_\Omega (u+1)^{q-2}u\nabla u\cdot\nabla v dx+\int_\Omega (u+1)^q dx.
\end{aligned}
\end{equation}
Define a nonnegative function \(F(\xi)=\int_0^\xi(\zeta+1)^{q-2}\zeta d\zeta.\) Given the nonnegativity of \(\partial_\nu v|_{\partial\Omega}\) and the increasing nature of the function \([0,\infty)\ni\zeta\mapsto(\zeta+1)^{q-2}\zeta\), it follows from integration by parts again, uniform boundedness of \(v\) as established in \eqref{prop_v}, and the Sobolev embedding that
\begin{equation}\label{eqn2c_4_1}
\begin{aligned}
-\int_\Omega (u+1)^{q-2}u\nabla u\cdot \nabla v dx&=\int_\Omega F(u)uv dx-\int_{\partial\Omega}F(u)\partial_\nu v dS\\
&\le M\int_\Omega (u+1)^{q-2}u^{3}dx\\
&\le M\sigma_n\int_0^R s^{n-1}(u(s,t)+1)^{q+1} ds.
\end{aligned}
\end{equation}
For convenience, we define \(w(s,t):=(u(s,t)+1)^\frac{q+m-1}{2}\) and \(z(s,t):=s^\frac{n-1}{2}w(s,t)\). By combining \eqref{eqn2c_4} and \eqref{eqn2c_4_1}, and then applying Young's inequality and H\"older inequality, we obtain
\begin{equation}\label{eqn2c_4_2}
\begin{aligned}
\frac{d}{dt}\norm{(u+1)}_q^q&+\norm{(u+1)}_q^q+\norm{\nabla w}_2^2\\
&\le C_5 \int_0^R s^{-\frac{(n-1)(2-m)}{q+m-1}}(z(s,t))^\frac{2(q+1)}{q+m-1}ds+C_5\\
&\le C_5\norm{z}_{L^\frac{2(q+1)k}{q+m-1}((0,R))}^\frac{2(q+1)}{q+m-1}\left(\int_0^R s^{-\frac{(n-1)(2-m)}{q+m-1}\cdot\frac{k}{k-1}} ds\right)^\frac{k-1}{k}+C_5
\end{aligned}
\end{equation}
with some \(C_5>0\), where \(k\) is given by \eqref{eqn2c_3}.
Since
\[\frac{(n-1)(2-m)}{q+m-1}\cdot\frac{k}{k-1}<1\]
due to \eqref{eqn2c_3}, it follows that
\begin{equation}\label{eqn2c_6_1}
\bigr{(}\int_0^R s^{-\frac{(n-1)(2-m)}{q+m-1}\cdot\frac{k}{k-1}} ds\bigr{)}^\frac{k-1}{k}\le C_6,
\end{equation}
where \(C_6>0\). Next, using a one-dimensional version of the Gagliardo-Nirenberg inequality, we observe that
\begin{equation}\label{eqn2c_5}
\norm{z}_{L^\frac{2(q+1)k}{q+m-1}((0,R))}\le C_7 \norm{\partial_s z}_{L^2((0,R))}^\theta\norm{z}_{L^\frac{2}{q+m-1}((0,R))}^{1-\theta}+\norm{z}_{L^\frac{2}{q+m-1}((0,R))}
\end{equation}
with some \(C_7>0\), where \[\theta=\frac{q+m-1-\frac{q+m-1}{(q+1)k}}{q+m}\in(0,1).\]
We note that 
\begin{equation*}
\frac{2(q+1)}{q+m-1}\theta=2\left(1-\frac{\frac{1}{k}+m-1}{q+m}\right)<2
\end{equation*}
due to \eqref{eqn2c_3}.
In addition, we have
\begin{equation*}
\norm{\partial_s z}_{L^2((0,R))}^2\le\int_0^R s^{
n-1}(\partial_sw(s,t))^2ds+\frac{(n-1)^2}{4}\int_0^R s^{n-3}w^2(s,t) ds,
\end{equation*}
which, with an application of Lemma~\ref{lem2b} with \(f(r)=w(r,t)\) and \(l=\frac{2}{q+m-1}\), entails  that 
\begin{equation}\label{eqn2c_6}
\norm{\partial_s z}_{L^2((0,R))}\le C_8(\norm{\nabla w}_2+1)
\end{equation}
with some \(C_8>0\).
Moreover, we use integration by parts to estimate
\begin{equation}\label{eqn2c_7}
\begin{aligned}
\norm{z}_{L^\frac{2}{q+m-1}((0,R))}^\frac{2}{q+m-1}&=\int_0^R s^\frac{n-1}{q+m-1} (u(s,t)+1)ds\\
&=\int_0^R s^{-n+1+\frac{n-1}{q+m-1}}\partial_s U(s,t)ds+R\\
&=C_9(\int_0^R s^{-n+\frac{n-1}{q+m-1}} U(s,t)ds+ 1),
\end{aligned}
\end{equation}
 where  \(C_9>0\). 
 Since \(\frac{n-1}{q+m-1}-\gamma>-1\) by \eqref{eqn2c_2}, applying \eqref{eqn2a_1} to \eqref{eqn2c_7} allows us to deduce for some \(C_{10}>0\)
\begin{equation}\label{eqn2c_8}
\norm{z}_{L^\frac{2}{q+m-1}((0,R))}^\frac{2}{q+m-1}\le C_9\eta\int_0^R s^{\frac{n-1}{q+m-1}-\gamma} ds+C_9\le C_{10},
\end{equation}
where \(\eta\) is a number defined in Lemma~\ref{lem2a}. Therefore, by combining \eqref{eqn2c_6} and \eqref{eqn2c_8} with \eqref{eqn2c_5}, we infer that
\begin{equation}\label{eqn2c_9}
\norm{z}_{L^\frac{2(q+1)k}{q+m-1}((0,R))}^\frac{2(q+1)}{q+m-1}\le C_{11} (\norm{\nabla w}_2^{\frac{2(q+1)}{q+m-1}\theta}+1),
\end{equation}
where \(C_{11}>0\). Consequently, by plugging \eqref{eqn2c_6_1} and  \eqref{eqn2c_9} into \eqref{eqn2c_4_2}, we can choose \(C_{12}>0\) using Young's inequality that 
\begin{align*}
\frac{d}{dt}\norm{(u+1)}_q^q+\norm{(u+1)}_q^q+\norm{\nabla w}_2^2&\le C_5 C_6 C_{11} (\norm{\nabla w}_2^{\frac{2(q+1)}{q+m-1}\theta}+1)+C_5\\
&\le\frac{1}{2}\norm{\nabla w}_2^2+C_{12}.
\end{align*}
Therefore, the function \(y(t):=\norm{(u(\cdot,t)+1)}_q^q\) satisfies 
\[y'(t)+y(t)\le C(q),\]
thereby leading to the result in \eqref{eqn2c_1}.
\end{proof}
\begin{proof}[\textit{Proof of Theorem~\ref{thm2}}]\quad
Let \(q_0=q\) satisfy \eqref{eqn2c_2}. By Lemma~\ref{lem2c}, there exists \(C>0\) such that \( \norm{u(\cdot,t)}_{q_0}\le C\) for all \(t\in(0,T_{max})\).
Let \(p\) be such that \(p\ge\max\left\{n-1,q_0-1\right\}\).
For fixed \(t\in(0,T_{max})\), we employ again \eqref{eqn2c_4} and \eqref{eqn2c_4_1} to obtain, for some \(C>0\),
\begin{equation}\label{proof_inq1}
\frac{d}{dt}\norm{(u+1)}_p^p+\norm{(u+1)}_p^p+\norm{\nabla(u+1)^\frac{p+m-1}{2}}_2^2
\le C(\norm{(u+1)}_{p+1}^{p+1}+1).
\end{equation}
By the Gagliardo-Nirenberg inequality, we see that for some \(C>0\)
\begin{equation}\label{proof_inq2}
\begin{aligned}
\norm{(u+1)}_{p+1}^{p+1}&=\norm{(u+1)^\frac{p+m-1}{2}}_\frac{2(p+1)}{p+m-1}^\frac{2(p+1)}{p+m-1}\\
&\le C\biggr{(}\norm{\nabla (u+1)^\frac{p+m-1}{2}}_2^{\frac{2(p+1)}{p+m-1}\theta}\norm{(u+1)^\frac{p+m-1}{2}}_\frac{2q_0}{p+m-1}^{\frac{2(p+1)}{p+m-1}(1-\theta)}\\&\quad+\norm{(u+1)^\frac{p+m-1}{2}}_\frac{2q_0}{p+m-1}^{\frac{2(p+1)}{p+m-1}}\biggr{)}\\
&= C(\norm{\nabla (u+1)^\frac{p+m-1}{2}}_2^{\frac{2(p+1)}{p+m-1}\theta}\|u+1\|_{q_{0}}^{(p+1)(1-\theta)}+\|u+1\|_{q_{0}}^{p+1}).
\end{aligned}
\end{equation}
Here, due to the condition on \(p\), we have \(\displaystyle\theta=\frac{\frac{p+m-1}{2q_0}-\frac{p+m-1}{2(p+1)}}{\frac{p+m-1}{n(2-m)}-\frac{1}{2}+\frac{1}{n}}\in(0,1).\)
Furthermore, since \(q_0\) fulfills \(q_0>\frac{(n-1)(2-m)}{m}+1-m>\frac{n}{2}(2-m)\), it follows that
\(
\frac{(p+1)\theta}{p+m-1}<1.
\)
Thus, substituting \eqref{proof_inq2} into \eqref{proof_inq1} and applying Young's inequality, we establish the \(L^p\)-boundedness of \(u\) for any \(p\ge\max\left\{n-1,q_0-1\right\}\).
Hence, similar to the proof of Theorem~\ref{thm1}, we can apply the standard iteration method to achieve the uniform boundedness of \(\norm{u(\cdot,t)}_\infty\) in \((0,T_{max})\).
\end{proof}

\section{Blow-up in a finite or infinite time for  \(0<m<\frac{2}{n}\)}
This section aims to detect the blow-up phenomena that occur due to the competition of diffusive and cross-diffusive dynamics. This will be verified through the analysis of a temporal evolution involving mass accumulation function \(U\).
\subsection{Estimates for \(v\) at the origin}
We begin by addressing the lower bound for \(\nabla \ln v\). For a detailed proof, we refer to \cite[Lemma 3.1]{wang_finite-time_2023}.
\begin{lem}\label{lem3a}
Let $(u,v)$ be a solution given in Proposition~\ref{prop_local} and $U$ be defined in \eqref{AUG25}. It holds that
\begin{equation}\label{eqn3a_1}
\partial_s (\ln v)(s,t)\ge\frac{s^{-n+1} U(s,t)}{1+\int_0^s\rho^{-n+1} U(\rho,t) d\rho}\quad\text{for all }\,(s,t)\in(0,R)\times(0,T_{max}).
\end{equation}
\end{lem}

As a result of Lemma~\ref{lem3a}, the value of \(v\) at the origin can be controlled by \(M\) with a quantity involving  \(U\).
\begin{lem}\label{lem3b}
Let $(u,v)$ be a solution given in Proposition~\ref{prop_local} and $U$ be defined in \eqref{AUG25}.
It holds that 
\begin{equation}\label{eqn3b_1}
v(0,t)\le\frac{M}{1+\psi(t)}\quad\text{ for all } t\in(0,T_{max}),
\end{equation}
where \(\psi(t):=\int_0^R\rho^{-n+1}U(\rho,t)d\rho\).
\end{lem}
\begin{proof}
From \eqref{eqn3a_1}, we have
\begin{equation*}
\begin{split}
\ln\frac{M}{v(0,t)}&=\int_0^R\partial_s(\ln v(s,t))ds\\
&\ge\int_0^R\partial_s(\ln(1+\int_0^s \rho^{-n+1} U(\rho,t)d\rho))ds\\
&=\ln(1+\int_0^R \rho^{-n+1}U(\rho,t)d\rho),
\end{split}
\end{equation*}
which immediately gives \eqref{eqn3b_1}.
\end{proof}
\subsection{ODI for a functional involving \(U\)}
We derive an ODI for \(\phi(t)\), which is crucial for our blow-up argument.
\begin{lem}\label{lem3c}
Let $(u,v)$ be a solution given in Proposition~\ref{prop_local} and $U$ be defined in \eqref{AUG25}. Assume that \(D\) satisfies \eqref{assume_D3} with some \(K_D>0\), and that
\[0<m<\frac{2}{n}.\]
Then, for any fixed \(\alpha\in(n-3,n(1-m)-1)\), a quantity
\[
\phi(t):=\int_0^R s^{-\alpha} U(s,t) ds\]
 belongs to \(C^0([0,T_{max}))\cap C^1((0,T_{max}))\), and satisfies 
\begin{equation}\label{eqn3c_1}
\partial_t \phi(t)\ge -C_1L^m + C_2\frac{(M-v(0,t))^2}{M}\end{equation}
with some \(C_1>0\) and  \(C_2>0\).
\end{lem}
\begin{proof}
Let \(t\in(0,T_{max})\). Obviously, \(\phi\) is well-defined for such \(\alpha\), and the regularity property of \(\phi\) is guaranteed via the dominated convergence theorem. In order to see \eqref{eqn3c_1}, we first define 
\begin{equation}\label{assume_tildeD}
\tilde{D}(\xi):=\int_0^\xi D(\sigma) d\sigma\quad\text{ for } \xi\ge 0.
\end{equation}
We recall that  \(U\)satisfies the following equation \eqref{eqn1c_2}  in \((0,R)\times (0,T_{max})\):
\begin{equation*}
\partial_t U(s,t)= s^{n-1}D(s^{-n+1} \partial_s U(s,t))\partial_s(s^{-n+1} \partial_s U(s,t))+\partial_s U(s,t) \partial_s v(s,t).
\end{equation*}
Then, \(\phi(t)\) is formulated as the following:
\begin{equation*}
\begin{split}
\partial_t\phi(t)=& \int_{0}^{R} s^{n-1-\alpha} \partial_s(\tilde{D}(s^{-n+1} \partial_s U(s,t))) ds \\
&+ \int_{0}^{R} s^{-\alpha} \partial_s U(s,t)\partial_s v(s,t) ds\\
=:&I_1(t)+I_2(t).
\end{split}
\end{equation*}
From the assumption \eqref{assume_D3} and \eqref{assume_tildeD}, we have
\[\tilde{D}(\xi)\le K_D\int_0^\xi(1+\sigma)^{m-1}d\sigma\le\frac{K_D}{m}\xi^m,\]
which entails, via integration by parts, that
\begin{equation}\label{eqn3c_2}
\begin{split}
I_1(t)&=-(n-1-\alpha)\int_0^R s^{n-2-\alpha}\tilde{D}(s^{-n+1} \partial_s U(s,t)) ds\\
&+\left\{s^{n-1-\alpha}\tilde{D}(s^{-n+1} \partial_s U(s,t))\right\}\bigg|_{s=0}^{R}\\
&\ge-\frac{K_D(n-1-\alpha)}{m}\int_0^R s^{n-2-\alpha-m(n-1)}(\partial_s U(s,t))^m ds,
\end{split}
\end{equation}
where we used the condition \(\alpha<n(1-m)-1\) to see that
\begin{equation*}
\lim_{s\rightarrow0}|s^{n-1-\alpha}\tilde{D}(s^{-n+1} \partial_s U(s,t))| \le \lim_{s\rightarrow0}s^{n-1-\alpha}\norm{u(\cdot,t)}_\infty^m=0.
\end{equation*}
Applying H\"older inequality and integration by parts, we further compute the rightmost term of  \eqref{eqn3c_2} as
\begin{equation}\label{eqn3c_3}
\begin{split}
\int_0^R &s^{n-2-\alpha-m(n-1)}(\partial_s U(s,t))^m ds\\
&\le\left(\int_0^R s^\frac{n-2-\alpha-m(n-1)}{1-m}ds\right)^{1-m}\left(\int_0^R \partial_s U(s,t) ds\right)^m.
\end{split}
\end{equation}
Here, the first integral on the right-hand side is finite because \(\alpha<n(1-m)-1\) implies 
\[
\frac{n-2-\alpha-m(n-1)}{1-m}>-1.
\] 
Hence, in light of \eqref{eqn3c_2} and \eqref{eqn3c_3}, one can find \(C_1>0\) such that
\[I_1(t)\ge-C_1(U(R,t)-U(0,t))^m=-C_1 L^m.\]
In order to estimate \(I_2(t)\), we first observe from \eqref{eqn_vr} that 
\begin{equation}\label{eqn3c_4}
\partial_s U(s,t)=\frac{\partial_s(s^{n-1}\partial_sv(s,t))}{v}\ge\partial_s(s^{n-1}\partial_s \ln v(s,t)).
\end{equation}
for all \(s\in(0,R)\). Using \eqref{eqn3c_4}, $v_{s}\ge0$, and integration by parts, we find 
\begin{equation}\label{eqn3c_5}
\begin{split}
I_2(t)&\ge\int_{0}^{R} s^{-\alpha} \partial_s(s^{n-1}\partial_s \ln v(s,t))\partial_s v(s,t) ds\\
&=\alpha\int_0^R s^{n-\alpha-2}\partial_s\ln v(s,t)\partial_s v(s,t)ds\\
&\quad-\int_0^R s^{n-\alpha-1}\partial_s\ln v(s,t)\partial_{ss} v(s,t)ds\\
&\quad+\left\{s^{n-\alpha-1}\partial_s\ln v(s,t)\partial_s v(s,t)\right\}\bigg|_{s=0}^{R}.
\end{split}
\end{equation}
 We employ \eqref{eqn_vr} again to see that
\begin{equation*}
\begin{split}
&-\int_0^R s^{n-\alpha-1}\partial_s\ln v(s,t)\partial_{ss} v(s,t)ds\\
&\quad=-\int_0^R s^{n-\alpha-1}\partial_s\ln v(s,t)\left(u(s,t)v(s,t)-\frac{n-1}{s}\partial_s v(s,t)\right) ds\\
&\quad=-I_2(t)+(n-1)\int_0^R s^{n-\alpha-2}\partial_s\ln v(s,t)\partial_s v(s,t)ds,
\end{split}
\end{equation*}
which yields, due to the nonnegativity of the last term in \eqref{eqn3c_5} and  \eqref{prop_v}, that
\begin{equation}\label{eqn3c_6}
\begin{split}
I_2(t)&\ge\frac{n+\alpha-1}{2}\int_0^R s^{n-\alpha-2}\partial_s \ln v(s,t)\partial_s v(s,t)ds\\
&\ge\frac{n+\alpha-1}{2M}\int_0^R s^{n-\alpha-2}(\partial_s v(s,t))^2 ds.
\end{split}
\end{equation}
We observe from H\"older inequality that 
\begin{equation}\label{eqn3c_7}
\begin{split}
(M-v(0,t))^2&=\left(\int_0^R \partial_s v(s,t)ds\right)^2\\
&\le\left(\int_0^R  s^{n-\alpha-2}(\partial_s v(s,t))^2 ds\right)\cdot\left(\int_0^R s^{-n+\alpha+2}ds\right),
\end{split}
\end{equation}
where the last integral is finite due to the condition \(n-3<\alpha\). Therefore, combining \eqref{eqn3c_6} and \eqref{eqn3c_7}, it follows that 
\begin{equation*}
I_2(t)\ge C_2\frac{(M-v(0,t))^2}{M}
\end{equation*}
with some \(C_2>0\). Hence, we can conclude \eqref{eqn3c_1}.
 \end{proof}
We are ready to present the proof of Theorem~\ref{thm3}. 

\begin{proof}[\textit{Proof of Theorem~\ref{thm3}.}]\quad
Let \(m\in(0,\frac{2}{n})\), and fix \(\alpha\) as in Lemma~\ref{lem3c}. Plugging \eqref{eqn3b_1} into \eqref{eqn3c_1} shows that 
\begin{equation}\label{eqn_proof1}
\partial_t \phi(t)\ge-C_1L^m+C_2 M\frac{\psi^2(t)}{(1+\psi(t))^2}
\end{equation}
for all \(t\in(0,T_{max})\), where \(C_1\) and \(C_2\) are constants specified in Lemma~\ref{lem3c}. For \(M>0\) and \(\xi\ge0\), let us define 
\[
\omega(\xi,M):=-2C_1L^m+C_2M\frac{\xi^2}{(1+\xi)^2},
\]
 and fix a constant depending \(u_0\)
\[
C^*\equiv C^*(u_0):= \frac{1}{2}R^{-n+\alpha+1}\phi(0).
\]
We choose \(M^*=M^*(u_0)>0\) such that 
\begin{equation*}
M^*\ge\frac{2C_1L^m}{C_2}\cdot\left(\frac{1+C^*}{C^*}\right)^2.
\end{equation*}
Since \(0\le x\mapsto x^2/(1+x)^2\) is increasing,  it holds that for \(\xi\ge C^*\) and \(M\ge M^*\)
\begin{equation}\label{eqn_proof3}
\omega(\xi,M)\ge -2C_1L^m+C_2M^*\left(\frac{C^*}{1+C^*}\right)^2\ge0.
\end{equation}
Upon fixing \(M\ge M^*\), and defining \(\phi\) and \(\psi\) as in Lemma~\ref{lem3c}, the continuity of \(\phi\) allows us to find \(T>0\) such that \(\phi(t)>\frac{\phi(0)}{2}\) for all \(t\in(0,T)\). This establishes the existence of
\[
T^*=\sup\left\{T\in(0,T_{max}) : \phi(t)>\frac{\phi(0)}{2}\,\text{ for all }\, t\in(0,T)\right\}.
\]
We show $T^*=T_{max}$. Indeed,
since it holds that
\[
\psi(t)\ge R^{-n+\alpha+1}\phi(t)\ge \frac{1}{2}R^{-n+\alpha+1}\phi(0)= C^*
\]
 for all \(t\in(0,T^*)\), it follows, in conjunction with \eqref{eqn_proof1} and \eqref{eqn_proof3}, that
 \begin{equation}\label{eqn_proof4}
\partial_t\phi(t)\ge \omega(\psi(t),M)+C_1L^m\ge C_1L^m
\end{equation}
for all \(t\in(0,T^*)\). Solving the ODI \eqref{eqn_proof4} yields 
\begin{equation}\label{eqn_proof5}
\phi(t)\ge \phi(0)+C_3t\quad\text{ for all }t\in(0,T^*)
\end{equation}
 with some \(C_3>0\), necessitating \(T^*=T_{max}\) because if \(T^*<T_{max}\), the continuity of \(\phi\) would imply \(\phi(T^*)=\frac{\phi(0)}{2}\), contradicting the growth established in \eqref{eqn_proof5}. 
Hence, in accordance with \eqref{eqn_proof5}, we obtain
\[
C_3t+\phi(0)\le\phi(t)\le R^{n-\alpha+1}\|u(\cdot,t)\|_\infty,
\]
which gives the desired results.
In particular, in the three-dimensional case, since \(\alpha\) can be set to satisfy \(0<\alpha<\min\left\{3(1-m)-1, 1\right\}\), it holds that 
\begin{equation*}
C_3t+\phi(0)\le\phi(t)=\int_0^R s^{-\alpha}U(s,t)ds\le\frac{L}{1-\alpha}R^{1-\alpha}
\end{equation*}
for all \(t\in(0,T_{max})\), which shows that \(T_{max}\) must be finite, as desired.
\end{proof}

\section*{Acknowledgement}
J. Ahn is supported by NRF grant No. RS-2024-00336346.
K. Kang is supported by NRF grant No. RS-2024-00336346. 
D. Kim is supported by the National Research Foundation of Korea(NRF) grant funded by the Korea government(MSIT)(grant No. 2022R1A4A1032094).

\bibliographystyle{abbrv}
\bibliography{reference_final}

\begin{thebibliography}{10}

\bibitem{ahn_regular_2023}
J.~Ahn, K.~Kang, and J.~Lee.
\newblock Regular solutions of chemotaxis-consumption systems involving
  tensor-valued sensitivities and {Robin} type boundary conditions.
\newblock {\em Mathematical Models and Methods in Applied Sciences},
  33(11):2337--2360, 2023.

\bibitem{ahn_global_2021}
J.~Ahn, K.~Kang, and C.~Yoon.
\newblock Global classical solutions for chemotaxis‐fluid systems in two
  dimensions.
\newblock {\em Mathematical Methods in the Applied Sciences}, 44(2):2254--2264,
  2021.

\bibitem{Ahn:2023aa}
J.~Ahn and M.~Winkler.
\newblock A critical exponent for blow-up in a two-dimensional
  chemotaxis-consumption system.
\newblock {\em Calculus of Variations and Partial Differential Equations},
  62(6):180, 2023.

\bibitem{alikakos1979lp}
N.~D. Alikakos.
\newblock Lp bounds of solutions of reaction-diffusion equations.
\newblock {\em Communications in Partial Differential Equations},
  4(8):827--868, 1979.

\bibitem{cieslak_finite-time_2008}
T.~Cieślak and M.~Winkler.
\newblock Finite-time blow-up in a quasilinear system of chemotaxis.
\newblock {\em Nonlinearity}, 21(5):1057--1076, 2008.

\bibitem{dombrowski_self-concentration_2004}
C.~Dombrowski, L.~Cisneros, S.~Chatkaew, R.~E. Goldstein, and J.~O. Kessler.
\newblock Self-concentration and large-scale coherence in bacterial dynamics.
\newblock {\em Physical Review Letters}, 93(9):098103, 2004.

\bibitem{fan_global_2017}
L.~Fan and H.-Y. Jin.
\newblock Global existence and asymptotic behavior to a chemotaxis system with
  consumption of chemoattractant in higher dimensions.
\newblock {\em Journal of Mathematical Physics}, 58(1):011503, 2017.

\bibitem{fuest_long-term_2021}
M.~Fuest, J.~Lankeit, and M.~Mizukami.
\newblock Long-term behaviour in a parabolic–elliptic
  chemotaxis–consumption model.
\newblock {\em Journal of Differential Equations}, 271:254--279, 2021.

\bibitem{jiang_eventual_2019}
J.~Jiang.
\newblock Eventual smoothness and exponential stabilization of global weak
  solutions to some chemotaxis systems.
\newblock {\em SIAM Journal on Mathematical Analysis}, 51(6):4604--4644, 2019.

\bibitem{keller_initiation_1970}
E.~F. Keller and L.~A. Segel.
\newblock Initiation of slime mold aggregation viewed as an instability.
\newblock {\em Journal of Theoretical Biology}, 26(3):399--415, 1970.

\bibitem{keller_model_1971}
E.~F. Keller and L.~A. Segel.
\newblock Model for chemotaxis.
\newblock {\em Journal of Theoretical Biology}, 30(2):225--234, 1971.

\bibitem{kim_global_2023}
D.~Kim.
\newblock Global solutions for chemotaxis-fluid systems with singular
  chemotactic sensitivity.
\newblock {\em Discrete and Continuous Dynamical Systems - B},
  28(10):5380--5395, 2023.

\bibitem{lankeit_locally_2017}
J.~Lankeit.
\newblock Locally bounded global solutions to a chemotaxis consumption model
  with singular sensitivity and nonlinear diffusion.
\newblock {\em Journal of Differential Equations}, 262(7):4052--4084, 2017.

\bibitem{lankeit_global_2020}
J.~Lankeit and G.~Viglialoro.
\newblock Global existence and boundedness of solutions to a
  chemotaxis-consumption model with singular sensitivity.
\newblock {\em Acta Applicandae Mathematicae}, 167(1):75--97, 2020.

\bibitem{lankeit_depleting_2023}
J.~Lankeit and M.~Winkler.
\newblock Depleting the signal: analysis of chemotaxis-consumption models—a
  survey.
\newblock {\em Studies in Applied Mathematics}, 151(4):1197--1229, 2023.

\bibitem{tao_boundedness_2011}
Y.~Tao.
\newblock Boundedness in a chemotaxis model with oxygen consumption by
  bacteria.
\newblock {\em Journal of Mathematical Analysis and Applications},
  381(2):521--529, 2011.

\bibitem{tao_boundedness_2012}
Y.~Tao and M.~Winkler.
\newblock Boundedness in a quasilinear parabolic-parabolic {Keller}-{Segel}
  system with subcritical sensitivity.
\newblock {\em Journal of Differential Equations}, 252(1):692--715, 2012.

\bibitem{tao_eventual_2012}
Y.~Tao and M.~Winkler.
\newblock Eventual smoothness and stabilization of large-data solutions in a
  three-dimensional chemotaxis system with consumption of chemoattractant.
\newblock {\em Journal of Differential Equations}, 252(3):2520--2543, 2012.

\bibitem{tao_global_2012}
Y.~Tao and M.~Winkler.
\newblock Global existence and boundedness in a {Keller}-{Segel}-{Stokes} model
  with arbitrary porous medium diffusion.
\newblock {\em Discrete and Continuous Dynamical Systems - A},
  32(5):1901--1914, 2012.

\bibitem{tuval_bacterial_2005}
I.~Tuval, L.~Cisneros, C.~Dombrowski, C.~W. Wolgemuth, J.~O. Kessler, and R.~E.
  Goldstein.
\newblock Bacterial swimming and oxygen transport near contact lines.
\newblock {\em Proceedings of the National Academy of Sciences},
  102(7):2277--2282, 2005.

\bibitem{viglialoro_global_2019}
G.~Viglialoro.
\newblock Global existence in a two-dimensional chemotaxis-consumption model
  with weakly singular sensitivity.
\newblock {\em Applied Mathematics Letters}, 91:121--127, 2019.

\bibitem{wang_global_2015}
L.~Wang, C.~Mu, K.~Lin, and J.~Zhao.
\newblock Global existence to a higher-dimensional quasilinear chemotaxis
  system with consumption of chemoattractant.
\newblock {\em Zeitschrift für angewandte Mathematik und Physik},
  66(4):1633--1648, 2015.

\bibitem{wang_boundedness_2014}
L.~Wang, C.~Mu, and S.~Zhou.
\newblock Boundedness in a parabolic-parabolic chemotaxis system with nonlinear
  diffusion.
\newblock {\em Zeitschrift für angewandte Mathematik und Physik},
  65(6):1137--1152, 2014.

\bibitem{wang_finite-time_2023}
Y.~Wang and M.~Winkler.
\newblock Finite-time blow-up in a repulsive chemotaxis-consumption system.
\newblock {\em Proceedings of the Royal Society of Edinburgh Section A:
  Mathematics}, 153(4):1150--1166, 2023.

\bibitem{winkler_global_2012}
M.~Winkler.
\newblock Global {large}-{data} {solutions} in a
  {chemotaxis}-({Navier}-){Stokes} {system} {modeling} {cellular} {swimming} in
  {fluid} {drops}.
\newblock {\em Communications in Partial Differential Equations},
  37(2):319--351, 2012.

\bibitem{winkler_two-dimensional_2016}
M.~Winkler.
\newblock The two-dimensional {Keller}–{Segel} system with singular
  sensitivity and signal absorption: {Global} large-data solutions and their
  relaxation properties.
\newblock {\em Mathematical Models and Methods in Applied Sciences},
  26(05):987--1024, 2016.

\bibitem{winkler_renormalized_2018}
M.~Winkler.
\newblock Renormalized radial large-data solutions to the higher-dimensional
  {Keller}–{Segel} system with singular sensitivity and signal absorption.
\newblock {\em Journal of Differential Equations}, 264(3):2310--2350, 2018.

\bibitem{winkler_approaching_2022}
M.~Winkler.
\newblock Approaching logarithmic singularities in quasilinear
  chemotaxis-consumption systems with signal-dependent sensitivities.
\newblock {\em Discrete and Continuous Dynamical Systems - B}, 27(11):6565,
  2022.

\bibitem{yan_global_2018}
J.~Yan and Y.~Li.
\newblock Global generalized solutions to a {Keller}–{Segel} system with
  nonlinear diffusion and singular sensitivity.
\newblock {\em Nonlinear Analysis}, 176:288--302, 2018.

\bibitem{yang_long_2024}
S.-O. Yang and J.~Ahn.
\newblock Long time asymptotics of small mass solutions for a
  chemotaxis-consumption system involving prescribed signal concentrations on
  the boundary.
\newblock {\em Nonlinear Analysis: Real World Applications}, 79:104129, 2024.

\bibitem{zhang_stabilization_2015}
Q.~Zhang and Y.~Li.
\newblock Stabilization and convergence rate in a chemotaxis system with
  consumption of chemoattractant.
\newblock {\em Journal of Mathematical Physics}, 56(8):081506, 2015.

\end{thebibliography}

\end{document}